\documentclass[a4paper,10pt]{amsart}
\usepackage[utf8]{inputenc}
\usepackage{amsxtra}
\usepackage{amsopn}
\usepackage{amsmath,amsthm,amssymb}
\usepackage{amscd}
\usepackage{mathtools}
\usepackage{amsfonts}
\usepackage{latexsym}
\usepackage{verbatim}
\usepackage{MnSymbol}
\usepackage{xcolor}
\usepackage{cases}
\usepackage{tikz-cd}
\usepackage{cases}

\newcommand{\la}{\llangle}
\newcommand{\ra}{\rrangle}

\newcommand{\Sol}{\mathbb{R} \times \mathrm{Sol}^3}
\newcommand{\tr}{\mathrm{tr}}
\let\c\overline
\theoremstyle{plain}
\newtheorem{theorem}{Theorem}[section]
\newtheorem*{theorem*}{Theorem}
\newtheorem*{theorem**}{Theorem \ref{thm-hyp}}
\newtheorem{definition}[theorem]{Definition}
\newtheorem{lemma}[theorem]{Lemma}
\newtheorem{proposition}[theorem]{Proposition}
\newtheorem{corollary}[theorem]{Corollary}
\newtheorem{remark}[theorem]{Remark}
\newtheorem{example}[theorem]{Example}
\newtheorem{question}[theorem]{Question}
\newtheorem*{question*}{Question}
\newtheorem*{mt*}{Main Theorem}
\newcommand\C{{\mathbb C}}

\newcommand\R{{\mathbb R}}
\newcommand\Z{{\mathbb Z}}

\newcommand{\de}[2]{\frac{\partial #1}{\partial #2}}
\newcommand\Span{{\hbox{\em Span}}}
\newcommand{\del}{{\partial}}
\newcommand{\delbar}{{\overline{\del}}}
\newcommand{\mubar}{\overline{\mu}}

\newcommand{\cinf}{\mathcal{C}^\infty}
\renewcommand{\H}{\mathcal{H}}

\DeclareMathOperator{\im}{im}
\DeclareMathOperator{\GL}{GL}

\DeclareMathOperator{\id}{id}

\let\c\overline
\let\phi\varphi
\let\Phi\varphi
\newcommand{\Cpf}{$\mathcal{C}^\infty$-pure-and-full}

\newcommand{\Cf}{$\mathcal{C}^\infty$-full}
\newcommand{\Cp}{$\mathcal{C}^\infty$-pure}

\title[Invariants of almost complex and almost K\"ahler manifolds]{Invariants of almost complex and almost K\"ahler manifolds}
\author{Tom Holt}
\address{Dipartimento di Scienze Matematiche, Fisiche e Informatiche\\
Unit\`{a} di Matematica e Informatica\\
Universit\`{a} degli Studi di Parma\\
Parco Area delle Scienze 53/A \\
43124 Parma, Italy}
\email{thomascharles.holt@unipr.it}
\author{Riccardo Piovani}
\email{riccardo.piovani@unipr.it}
\author{Adriano Tomassini}
\email{adriano.tomassini@unipr.it}

\thanks{\newline 
The second author is partially supported by GNSAGA of INdAM.
The third author is partially supported by the Project PRIN 2017 ``Real and Complex Manifolds: Topology, Geometry and holomorphic dynamics'' and by GNSAGA of INdAM}

\begin{document}


\begin{abstract} 
Let $(M^{2n},J)$ be a compact almost complex manifold.
The almost complex invariant $h^{p,q}_J$ is defined as the complex dimension of the cohomology space $\left\{\left[\alpha\right]\in H^{p+q}_{dR}(M^{2n};\C) \,\vert\,\alpha\in A^{p,q}(M^{2n}),\, d\alpha = 0 \right\}$. When $2n=4$, it has many interesting properties. Endow $(M^{2n},J)$ with an almost Hermitian metric $g$. The number $h^{p,q}_d$, i.e., the complex dimension of the space of Hodge-de Rham harmonic $(p,q)$-forms, is almost K\"ahler invariant when $2n=4$. In this paper we study the relationship between $h^{p,q}_J$ and $h^{p,q}_d$ in dimension $2n\ge4$. We prove $h^{n,0}_J=0$ if $J$ is non integrable and show that $h^{p,0}_d$ is almost K\"ahler invariant. If $M^{2n}$ is a compact quotient of a completely solvable Lie group and $(J,g,\omega)$ is left invariant, we find information also on $h^{1,1}_d$. Finally we study the \Cp\ and \Cf\ properties of $J$ on $n$-forms for the special dimension $2n=4m$.
\end{abstract}

\maketitle

\section{Introduction}\label{sec intro}
An almost complex manifold is a pair $(M,J)$,  where $M$ is a smooth $2n$-dimensional manifold and $J$ is a smooth $(1,1)$-tensor field satisfying $J^2=-\id$. If the Nijenhuis tensor of $J$ vanishes then the almost complex structure $J$ is said to be integrable and $(M,J)$ has the structure of a complex manifold by the Newlander-Nirenberg Theorem. For a complex manifold $(M,J)$,  the exterior derivative $d$ splits as $d=\del+\delbar$, where $\del^2=0$, $\delbar^2=0$ and $\del\delbar+\delbar\del =0$, thus allowing us to consider the Dolbeault cohomology groups defined as $H_\delbar^{\bullet,\bullet}(M)=\frac{\ker \delbar}{\im \delbar}$. Therefore, their dimensions $h^{p,q}$ are complex invariants. 

An almost Hermitian metric on an almost complex manifold $(M,J)$ is a Riemannian metric $g$ such that $J$ acts as an isometry, also called $J$-Hermitian, and its fundamental form $\omega$ is defined as $\omega(X,Y)=g(JX,Y)$.
An (almost) Hermitian manifold is a quadruple $(M,J,g,\omega)$ where $(M,J)$ is an (almost) complex manifold and  $g$ is a $J$-Hermitian metric with fundamental form $\omega$; if $d\omega=0$, the metric is called (almost) K\"ahler and $(M,J,g,\omega)$ is an (almost) K\"ahler manifold.

Compact K\"ahler manifolds satisfy many remarkable metric and cohomological properties: indeed, as a consequence of the K\"ahler identites, the  Hodge-de Rham Laplacian $\Delta_d=dd^*+d^*d$, the $\del$-Laplacian $\Delta_{\del}=\del\del^*+\del^*\del$, and the $\delbar$-Laplacian $\Delta_{\delbar}=\delbar\,\delbar^*+\delbar^*\delbar$ satisfy 
$$
\Delta_d=2\Delta_\del=2\Delta_\delbar
$$
so, by denoting 
$$
\mathcal{H}^{\bullet,\bullet}_{d}=\ker \Delta_d\vert_{A^{\bullet,\bullet}(M)},
\quad \mathcal{H}^{\bullet,\bullet}_{\del}=\ker \Delta_{\del}\vert_{A^{\bullet,\bullet}(M)},\quad
\mathcal{H}^{\bullet,\bullet}_{\delbar}=\ker \Delta_{\delbar}\vert_{A^{\bullet,\bullet}(M)},
$$
it holds that $\mathcal{H}^{\bullet,\bullet}_{d}=\mathcal{H}^{\bullet,\bullet}_{\del}=\mathcal{H}^{\bullet,\bullet}_{\delbar}$. Recall that they have finite dimension since they are the kernels of self adjoint elliptic differential operators. Furthermore, via the Hodge isomorphism $\mathcal{H}^{\bullet,\bullet}_{\delbar}\simeq H^{\bullet,\bullet}_\delbar(M)$, the complex de Rham $k$-cohomology groups decompose as the direct sum of Dolbeault cohomology groups, that is 
$$
H^{k}_{dR}(M;\C)\simeq \bigoplus_{p+q=k}H^{p,q}_\delbar (M).
$$
In particular, $b^{k}=\displaystyle\sum_{p+q=k}h^{p,q}$. 

For a compact almost Hermitian manifold $(M,J,g,\omega)$ the $\delbar$-Laplacian $\Delta_{\delbar}$ can still be defined and it turns out that it is still an elliptic differential operator. In \cite[Problem 20]{Hi}, Kodaira and Spencer asked the following: given a non integrable almost complex structure $J$ on a compact manifold, choose an almost Hermitian metric $g$ and set $h^{p,q}_\delbar=\dim_\C\mathcal{H}^{p,q}_{\delbar}$. Do the  numbers $h^{p,q}_\delbar=\dim_\C\mathcal{H}^{p,q}_{\delbar}$ depend on the almost Hermitian metric $g$?

Very recently, the first author and W. Zhang \cite{HZ} proved that $h^{0,1}_\delbar$ can vary with different choices of the almost Hermitian metric on the Kodaira-Thurston manifold. Therefore, $h^{p,q}_\delbar$ is not an almost complex invariant. More precisely, $h^{p,q}_\delbar$ is not even almost K\"ahler invariant, meaning that it can assume different values for different almost K\"ahler metrics \cite{HZ2}. For other results concerning the study of $h^{p,q}_\delbar$ see \cite{TT, Ho, PT5, zhang-parma, P, tardini-tomassini-dim6} and the references therein. See also \cite{cattaneo-tardini-tomassini, piovani-tardini, HP} for the study of the relation between the spaces of harmonic forms and the primitive decomposition on compact almost K\"ahler manifolds.

In the present paper, given an almost complex manifold $(M,J)$, we are interested in studying almost complex invariants and objects which are almost K\"ahler (metric) invariant, i.e., objects which are independent of the choice of almost K\"ahler metrics on $(M,J)$.

Draghici, Li and Zhang \cite{DLZ}, starting with a $2n$-dimensional compact almost complex manifold $(M,J)$, defined the spaces
$$
H^{p,q}_J(M)= \left\{\left[\alpha\right]\in H^{p+q}_{dR}(M;\C) \,\,\,\vert\,\,\,
\alpha\in A^{p,q}(M),\, d\alpha = 0 \right\}.
$$
and the numbers $h^{p,q}_J=\dim_{\C} H^{p,q}_J(M)$, which are almost complex invariants. In Theorem \ref{n0}, by making use of the unique continuation theorem for harmonic differential forms, we prove that $h^{n,0}_J=0$ if $J$ is non integrable. This extends \cite[Lemma 2.12]{DLZ} to the $2n$-dimensional case.

A natural candidate to be almost K\"ahler invariant is
\[
h^{p,q}_d:=\dim_\C \mathcal{H}^{p,q}_d,
\]
which is bounded from above by
\begin{equation}\label{eq-ineq-intro}
h^{p,q}_d\le h^{p,q}_J.
\end{equation}
Note that in the integrable case the numbers $h^{p,q}_d$ do not depend on the choice of the K\"ahler metric by Hodge theory.
Recently, by combining \cite[Lemma 2.12]{DLZ} and \cite[Proposition 6.1]{HZ}, it was proven in \cite[Corollary 5.9]{CW} that, given any compact almost K\"ahler $4$-dimensional manifold, $h^{p,q}_d$ is almost K\"ahler invariant.
As a first result we prove that, given any compact $2n$-dimensional almost K\"ahler manifold, the numbers $h^{p,0}_d$ are almost K\"ahler invariant since $h^{p,0}_d=h^{p,0}_J$.
Furthermore, by \eqref{eq-ineq-intro} the value of $h^{p,0}_d$ obtained when the metric is almost K\"ahler is maximal over all almost Hermitian metrics.

In Corollary \ref{cor-nilm} we also prove that $h^{1,1}_d=h^{1,1}_J$ on any $2n$-dimensional compact quotient of a completely solvable Lie group endowed with a left invariant almost K\"ahler metric. In particular $h^{1,1}_d$ is invariant of left invariant almost K\"ahler metrics in this setting, and by \eqref{eq-ineq-intro} it is maximal over all almost Hermitian metrics. Furthermore, dropping the almost K\"ahler assumption and restricting to dimension 4 in the same setting as above, in Theorem \ref{thm-gaud-bc-inj} we obtain that $h^{1,1}_J\in\{b^-,b^-+1\}$, and in Corollary \ref{cor-compl-solv-non-kahl} we characterise both cases for the value of $h^{1,1}_J$. Recall that $b^-$ is the real dimension of the space of anti self dual harmonic $2$-forms.

In a private communication W. Zhang asked whether, in dimension $4$, $h^{1,0}_d$ is an almost complex invariant or may depend on the almost Hermitian metric. 
In Theorem \ref{thm-10d} we show that $h^{1,0}_d$ is not an almost complex invariant. In order to prove this, we study $h^{1,0}_d$ on 4-dimensional non K\"ahler solvmanifolds. Specifically, we find examples for which the value of $h^{1,0}_d$ is fixed for any choice of left invariant almost Hermitian metric, but differs for some non left invariant almost Hermitian metrics. 

The main results of the paper can be summarised in the following Theorem.
\begin{theorem}
Let $(M^{2n},J,g,\omega)$ be a compact almost Hermitian $2n$-manifold. Then:
\begin{itemize}
\item $h^{n,0}_J=0$;
\item if $g$ is almost K\"ahler, then $h^{p,0}_d=h^{p,0}_J$;
\item if $M^{2n}=\Gamma\backslash G$ is a compact quotient of a completely solvable Lie group $G$ and $(J,g,\omega)$ is left invariant:
\begin{itemize}
\item if $g$ is almost K\"ahler, then $h^{1,1}_d=h^{1,1}_J\le h^{n-1,n-1}_J$;
\item if $2n=4$ and $g$ is almost K\"ahler, then $h^{1,1}_J=b^-+1$;
\item if $2n=4$, $h^{1,1}_J\in\{b^-,b^-+1\}$;
\item if $2n=4$, and given any closed $(1,1)$-form $c\omega+\gamma$ with $0\ne c\in\C$ and $*\gamma=-\gamma$, then $h^{1,1}_J=b^-$ iff $c\omega+\gamma$ is exact.
\end{itemize} 
\end{itemize}
\end{theorem}

In Section \ref{sec examples} we include computations of harmonic $(p,q)$-forms on a six dimensional nilmanifold endowed with a left invariant almost Hermitian structure. By direct computations and general results we prove that in the example $h^{p,q}_d$ does not depend on the choice of the left invariant almost K\"ahler metric, and the value of $h^{p,q}_d$ given by almost K\"ahler metrics is maximal over all almost Hermitian metrics even if $h^{2,1}_d<h^{2,1}_J$. We would like to remark that, by making use of an automated computer program, we were able to construct a very large number of examples of nilmanifolds which all show similar behaviour. This suggests that there is a chance that $h^{p,q}_d$ is actually almost K\"ahler invariant, at least for some special class of compact almost complex manifolds. We state the following question in the more general case, but remark that even an answer for the simpler case of nilmanifolds with left invariant almost K\"ahler metrics would be interesting.
\begin{question}
Let $(M,J,g,\omega)$ be a compact almost K\"ahler manifold.\\
Is $h^{p,q}_d(M,J,g,\omega)$ independent of the choice of the almost K\"ahler metric?\\ If so, is this value of $h^{p,q}_d$ maximal over all almost Hermitian metrics?
\end{question}

Finally, in \cite{DLZ}, Draghici, Li and Zhang introduced and studied \Cp\ and \Cf\ almost complex structures. See section \ref{sec cohomology} for the definitions. In particular, they proved that any almost complex structure on a compact $4$-dimensional manifold is \Cpf.
A natural generalisation of their notion of \Cp\ and \Cf\ to the higher dimension is the notion of \Cp\ and \Cf\ in degree $2n$ on compact almost Hermitian manifolds of dimension $4n$. Indeed, on $2n$-forms both the almost complex structure and the Hodge $*$ operator act as involutions, i.e., $J^2=*^2=1$, yielding decompositions of the space of $2n$-forms into the $\pm1$ eigenspaces of both $J$ and $*$. This looks promising for attempts to generalise the results of \cite{DLZ}. In this paper we find an 8-dimensional nilmanifold which is not \Cp\ nor \Cf\ in degree $4$, see Example \ref{example fino tomassini}. However, on any $4n$-dimensional compact almost Hermitian manifold, if we restrict to the cohomology spaces
\[
H^{2n}_{g\pm}:=\{a\in H^{2n}_{dR}\,|\,a=[\alpha], *\alpha=\pm\alpha\},
\]
it holds that
\[
H^{2n}_{dR}=H^{2n}_{g+}\oplus H^{2n}_{g-}
\]
and we are able to prove that $J$ is \Cp\ in degree $2n$ when restricted to the two spaces $H^{2n}_{g+}$ and $H^{2n}_{g-}$ (see Theorem \ref{thm 4n pure}).

The structure of this paper is as follows. In section \ref{sec preliminaries} we recall some preliminary definitions and results involving almost complex manifolds and harmonic forms. In section \ref{sec cohomology} we give the definitions of \Cp\ and \Cf\ almost complex structures following \cite{DLZ}. In section \ref{sec 4 dim alm cplx inv} we study the almost complex invariance of $h^{p,q}_d$ in dimension 4.
In section \ref{sec-invariants-p0} we study the numbers $h^{p,0}_d$ in higher dimension. In section \ref{sec quotient lie 11} we study the relationship between $h^{1,1}_d$ and $h^{1,1}_J$ on compact quotients of Lie groups. In section \ref{sec examples}, through explicit examples, we prove that the inequality \eqref{eq-ineq-intro} can be strict, and show that for the equality $h^{1,1}_d=h^{1,1}_J$ to hold on 4-nilmanifolds, the assumption that the almost K\"ahler structure is left invariant is necessary. Finally, in section \ref{sec 4n} we study the properties on an almost complex structure of being \Cp\ or \Cf\ in degree $2n$ on a compact almost Hermitian manifold of dimension $4n$.

\medskip\medskip
\noindent{\em Acknowledgments.}
The authors are sincerely grateful to Weiyi Zhang and Lorenzo Sillari for useful discussions.

\section{Preliminaries}\label{sec preliminaries}
Let $\left(M,J\right)$ be an almost complex compact $2n$-manifold. Then the almost complex structure $J$ can be extended to an endomorphism of $TM^{\C}=TM\otimes\C$, still denoted by $J$, so that $J$ has eigenvalues $+i$ and $-i$, with the corresponding eigenbundles denoted, as usual, by $T^{1,0}M$ and $T^{0,1}M$, respectively.
$$
T^{1,0}M=\{X-iJX\,\,\,\vert\,\,\, X\in TM\},
$$
$$
T^{0,1}M=\{X+iJX\,\,\,\vert\,\,\, X\in TM\}.
$$
Accordingly, $J$ induces a decomposition of the bundle of complex forms into the direct sum of the bundle of $(p,q)$-forms, namely 
$$\Lambda^{r}M\otimes\C=\bigoplus_{p+q=r}\Lambda^{p,q}M,
$$
where 
$$\Lambda^{p,q}M=\Lambda^pT^{1,0}M\wedge\Lambda^qT^{0,1}M$$ 
is the bundle of $(p,q)$-{\em forms} on $M$. We will denote by $A^r_{\C}(M)=\Gamma(M,\Lambda^{r}M\otimes\C)$ the space of smooth sections of $\Lambda^{r}M\otimes\C$, i.e., the space of {\em complex r-forms} and by $A^{p,q}(M)=\Gamma(M,\Lambda^{p,q}M)$ the space of smooth sections of $\Lambda^{p,q}M$, i.e., the space of $(p,q)$-{\em forms} on $M$. The exterior derivative $d_{\vert A^{p,q}}$ decomposes as
$$
d=\mu+\del+\delbar +\overline{\mu}, 
$$ 
where 
$$\del=\pi^{p+1,q}\circ d_{\vert A^{p,q}},\qquad \delbar=\pi^{p,q+1}\circ d_{\vert A^{p,q}}$$ 
are first order differential operators and 
$$\mu=\pi^{p+2,q-1}\circ d_{\vert A^{p,q}},\qquad \overline{\mu}=\pi^{p-1,q+2}\circ d_{\vert A^{p,q}}
$$ 
are zeroth order operators, $\pi^{p,q}:A^{p+q}_{\C}(M)\to A^{p,q}(M)$ here denotes the natural projection.
$J$ is called \emph{integrable} when the almost complex manifold $(M,J)$ is induced by the structure of a complex manifold, and this occurs iff $\mubar\equiv0$.

Let $g$ be an almost Hermitian metric on $(M,J)$, that is a Riemannian metric on $M$ such that $J$ acts as a $g$-isometry. We will denote by $\omega$ the {\em fundamental form} of $g$, i.e., the real $2$-form on $M$ defined as 
$$
\omega(X,Y)=g(JX,Y),
$$ 
for any pair of vector fields $X$, $Y$ on $M$.

If $*:A^{p,q}M\to A^{n-q,n-p}M$ denotes the complex linear Hodge operator, one can consider the following second order elliptic differential operators defined respectively as
\begin{align*}
\Delta_d=&d d^*+d^*d,\\
\Delta_{\delbar}=& \delbar\,\delbar^*+\delbar^*\delbar.
\end{align*}
We denote by
\begin{align*}
\mathcal{H}^{p,q}_d=&\{\alpha\in A^{p,q}(M)\,:\, \Delta_d\alpha=0\},\\
\mathcal{H}^{p,q}_{\delbar}=&\{\alpha\in A^{p,q}(M)\,:\, \Delta_{\delbar}\alpha=0\}
\end{align*}
respectively, the spaces of $d$ and $\delbar$ harmonic $(p,q)$-forms.
If $M$ is compact, $\mathcal{H}^{p,q}_d$ and  $\mathcal{H}^{p,q}_{\delbar}$ are complex finite dimensional vector spaces and 
\begin{eqnarray*}
\alpha\in\mathcal{H}^{p,q}_d&\iff &
\left\{
\begin{array}{l}
d\alpha=0,\\
d*\alpha=0,
\end{array}
\right.
\\
\alpha\in\mathcal{H}^{p,q}_{\delbar}&\iff &
\left\{
\begin{array}{l}
\delbar\alpha=0,\\
\del*\alpha=0.
\end{array}
\right.
\end{eqnarray*}

Note that the following equalities hold for the space of $d$-harmonic $(p,q)$-forms.
\begin{equation}\label{eq conj harm}
\overline{\H^{p,q}_d}=\H^{q,p}_d
\end{equation}
\begin{equation}\label{eq star harm}
*\H^{p,q}_d=\H^{n-q,n-p}_d
\end{equation}
They are induced, respectively, by complex conjugation and the Hodge $*$ operator.
This implies the following chain of equalities for the dimension $h^{p,q}_d$:
\begin{equation}\label{eq numbers}
h^{p,q}_d=h^{q,p}_d=h^{n-p,n-q}_d=h^{n-q,n-p}_d.
\end{equation}

In \cite{PT4} the second and third author introduced the following fourth order elliptic and formally self adjoint operator on an almost Hermitian manifold
\begin{equation*}
\tilde\Delta_{BC}=
\del\delbar\delbar^*\del^*+
\delbar^*\del^*\del\delbar+\del^*\delbar\delbar^*\del+\delbar^*\del\del^*\delbar
+\del^*\del+\delbar^*\delbar,
\end{equation*}
which is called  \emph{Bott-Chern Laplacian} as in the integrable case.
If $M$ is compact, then for all $\alpha\in A^{p,q}$ it holds that
\begin{equation*}
\tilde\Delta_{BC}\alpha=0\ \iff\  \del\alpha=0,\ \delbar\alpha=0,\ \del\delbar*\alpha=0,
\end{equation*}
and this characterises the space of Bott-Chern harmonic $(p,q)$-forms 
\[
\mathcal{H}^{p,q}_{BC}=\{\alpha\in A^{p,q}(M)\,:\, \tilde\Delta_{BC}\alpha=0\}.
\]

We end this section by recalling the notion of primitive forms.
We denote by
$$
L:\Lambda^kM\to\Lambda^{k+2}M\,,\quad \alpha\mapsto\omega\wedge\alpha
$$
the Lefschetz operator and by
$$
\Lambda:\Lambda^kM\to\Lambda^{k-2}M\,,\quad \Lambda=*^{-1}L*
$$
its adjoint.
We recall that the map $L^h:\Lambda^kM\to\Lambda^{k+2h}M$ is injective for $h+k\le n$ and is surjective for $h+k\ge n$. 
A differential $k$-form $\alpha$ on $M$, for $k\leq n$, is said to be {\em primitive} if $\Lambda\alpha=0$, or equivalently if 
$$
L^{n-k+1}\alpha=0.
$$
 Then we have the following vector bundle decomposition (see e.g., \cite[p. 26, Th\'eor\`eme 3]{weil})
\begin{equation*}\label{eq-prim-dec-forms}
\Lambda^kM=\bigoplus_{r\geq\max(k-n,0)}L^r(P^{k-2r}M),
\end{equation*}
where we use
$$
P^{s}M:=\ker\big(\Lambda:\Lambda^{s}M\to\Lambda^{s-2}M\big)
$$
to denote the bundle of primitive $s$-forms. 
For any given $\beta\in P^kM$, we have the following formula (cf. \cite[p. 23, Th\'eor\`eme 2]{weil}) involving the Hodge $*$ operator and the Lefschetz operator
\begin{equation}\label{eq-*-primitive}
*L^r\beta=(-1)^{\frac{k(k+1)}{2}}\frac{r!}{(n-k-r)!}L^{n-k-r}J\beta.
\end{equation}

Furthermore, the decomposition above is compatible with the bidegree decomposition on the bundle of complex $k$-forms $\Lambda_\C^kM$ induced by $J$, that is 
$$
P^k M \otimes \C=\bigoplus_{p+q=k}P^{p,q}M,
$$
where 
$$
P^{p,q}M= \left(P^k M \otimes \C \right) \cap\Lambda^{p,q}M.
$$
In fact, we have
\begin{equation}\label{eq-prim-dec-forms-2}
\Lambda^{p,q}M=\bigoplus_{r\geq\max(p+q-n,0)}L^r(P^{p-r,q-r}M).
\end{equation}
Finally, let us set $P^s:=\Gamma(M,P^sM)$ and $P^{p,q}:=\Gamma(M,P^{p,q}M)$.

\section{Complex cohomology groups of almost complex manifolds}\label{sec cohomology}
Let $(M,J)$ be an almost complex manifold of dimension $2n$.
Following T.-J. Li and W. Zhang \cite[Section 2]{LZ} and T. Draghici, T.-J. Li and W. Zhang \cite{DLZ}, we recall the notion of de Rham cohomology with pure-type representatives. 
Define the cohomology space
$$
H^{p,q}_J(M)= \left\{\left[\alpha\right]\in H^{p+q}_{dR}(M;\C) \,\,\,\vert\,\,\,
\alpha\in A^{p,q}(M), d\alpha =0 \right\}
$$
and denote its complex dimension by
$$
h^{p,q}_J=\dim_{\C} H^{p,q}_J(M),
$$
which is always trivially bounded from above by $b^{p+q}$.
Given an almost Hermitian metric, the injection of the space of harmonic $(p,q)$-forms $\H^{p,q}_d$ into the cohomology spaces $H^{p,q}_J$ provides a metric dependent lower bound for $h^{p,q}_J$.

\begin{lemma}\label{lemma-injection}
Let $(M,J,g,\omega)$ be a compact almost Hermitian manifold. The map
\begin{align*}
j:&\H^{p,q}_d\hookrightarrow H^{p,q}_J\\
&\alpha\mapsto[\alpha]_{dR}
\end{align*}
is injective and thus $h^{p,q}_d\le h^{p,q}_J$.
\end{lemma}
\begin{proof}
Since harmonic and exact forms are orthogonal, note that if $\alpha=d\gamma$, then $\alpha=0$. Thus $j$ is injective.
\end{proof}

In Example \ref{ex fino tomassini} we will build a left invariant almost K\"ahler structure on a 6-dimensional nilmanifold where the previous inequality is strict for the bidegree $({2,1})$.

We will now define what it means for an almost complex structure to be \Cp\ and \Cf\ in degree $2k$. Note that an almost complex structure $J$ satisfies $J^2=(-1)^r$ on $r$-forms. When $r=2k$, this means that $J$ induces a decomposition of $A^{2k}(M)$ into its two $\pm1$-eigenbundles. 
We set
\[
H^{2k}_{J+}(M)=\{[\alpha]\in H^{2k}_{dR}(M;\R)\,|\,J\alpha=\alpha\}
\]
and
\[
H^{2k}_{J-}(M)=\{[\alpha]\in H^{2k}_{dR}(M;\R)\,|\,J\alpha=-\alpha\}.
\]
Let $M$ be a $2n$-dimensional manifold.
Now we give the following definition.
\begin{definition}
An almost complex structure $J$ on $M$ is said to be 
\begin{itemize}
 \item \emph{\Cp\ in degree $2k$} if
$$ H^{2k}_{J+}(M)\;\cap\; H^{2k}_{J-}(M) \;=\; \left\{\left[0\right]\right\} \;; $$
 \item \emph{\Cf\ in degree $2k$}  if
$$
H^{2k}_{J+}(M)\;+\; H^{2k}_{J-}(M) \;=\; H^{2k}_{dR}(M;\R) \;;
$$
 \item \emph{\Cpf\ in degree $2k$} if it is both \Cp\ and \Cf\ in degree $2k$, i.e., if the following decomposition holds:
$$
H^{2k}_{J+}(M)\,\oplus\, H^{2k}_{J-}(M) \,=\, H^{2k}_{dR}(M;\R) \;.
$$
\end{itemize}
\end{definition}

\begin{remark}
In practice, the above definition means that the elements $[\alpha] \in H^{2k}_{J+}(M)$ and $[\beta] \in H^{2k}_{J-}(M)$ are exactly the cohomology classes that can be represented by the real valued $2k$-forms
$$\alpha \in A^{2k}(M) \cap \bigoplus_{j \in 2\mathbb Z}A^{k+j,k-j}(M), $$
$$\beta \in A^{2k}(M) \cap \bigoplus_{j \in 2\mathbb Z +1}A^{k+j,k-j}(M). $$
\end{remark}

When $2k=2$, we will simply say \Cp, \Cf{} and \Cpf.
Draghici, Li and Zhang proved in \cite{DLZ} that every almost complex structure on a compact 4-manifold is \Cpf. In this case we set 
\[
h^+_J:=\dim_\R H^{2}_{J+}(M),\ \ \ h^-_J:=\dim_\R H^{2}_{J-}(M).
\]
Note that $h^+_J=h^{1,1}_J$ and $b^2=h^+_J+h^-_J$.

\section{$h^{p,q}_d$ on $4$-dimensional compact almost complex manifolds}\label{sec 4 dim alm cplx inv}

Cirici and Wilson \cite[Corollary 5.9]{CW} proved that the numbers $h^{p,q}_d$ are almost K\"ahler invariant for compact almost complex $4$-manifolds, i.e., they do not change with different choices of almost K\"ahler metrics on a given compact almost complex 4-manifold.

On the other hand, Weiyi Zhang brought to our attention the question of the almost complex invariance of the numbers $h^{p,q}_d$ in dimension 4. 
On a compact almost Hermitian 4-manifold, by \eqref{eq numbers}, the only numbers $h^{p,q}_d$ that we have to study are
\[
h^{1,0}_d,\ \ \ h^{2,0}_d,\ \ \ h^{1,1}_d.
\]
Let us recall what is already known about the almost complex invariance of these numbers. 
 
 By \cite[Corollary 4.5]{PT4} or \cite[Theorem 3.4]{Ho} we know that $h^{1,1}_d=b^-+1$ if and only if the metric is globally conformal to an almost K\"ahler metric, otherwise $h^{1,1}_d=b^-$. 
Therefore $h^{1,1}_d$ is almost K\"ahler invariant, but it is not an almost complex invariant. More generally, note that $h^{1,1}_d$ is a symplectic invariant.
 
On the other hand, taking into account that $\H^{2,0}_d$ injects into $H^{2,0}_J$, by \cite[Lemma 2.12]{DLZ} we know that $\H^{2,0}_d$ is zero when the almost complex structure is not integrable, and by \cite[Lemma 2.12, Proposition 2.17]{DLZ} the space $\H^{2,0}_d$ is just $\H^{2,0}_{\delbar}$ when the almost complex structure is integrable. In both cases, note that $h^{2,0}_d$ is an almost complex invariant.

In this section we study the behaviour of $h^{1,0}_{d}$ on compact almost Hermitian 4-manifolds, thus giving a complete picture of the almost complex invariance of $h^{p,q}_d$ in dimension 4. In fact, in this section we will give examples proving the following theorem. 
\begin{theorem}\label{thm-10d}
The number $h^{1,0}_d$ is not an almost complex invariant on compact 4-dimensional manifolds, nor is it even a complex invariant.
\end{theorem}
In particular, we focus on the case of compact non K\"ahler solvmanifolds in dimension 4 as listed in \cite[Table 6.1]{B}, giving a full account of the possible values of $h^{1,0}_d$ in all cases.

We have the following simple upper bound on $h^{1,0}_d$.
\begin{lemma}\label{lemma 1,0}
On any compact almost Hermitian manifold
$$h^{1,0}_d \leq \frac{b^1}{2}. $$
\end{lemma}
\begin{proof}
The space of harmonic forms includes both the harmonic $(1,0)$-forms and the harmonic $(0,1)$-forms, therefore we have the inequality
$$h^{1,0}_d + h^{0,1}_d \leq b^1.$$
The lemma then follows by equating $h^{1,0}_d$ with $h^{0,1}_d$ as in \eqref{eq numbers}.
\end{proof}

While proving Theorem \ref{thm-10d}, we also observe the following behaviour of the number $h^{1,0}_d$.
\begin{theorem}
    For each class of 4-dimensional non K\"ahler solvmanifold (cf. \cite[Table 6.1]{B}), $h^{1,0}_d$ can be made to take every integer value between 0 and $\frac{b^1}{2}$ by just varying the almost Hermitian metric. 
\end{theorem}

By Lemma \ref{lemma 1,0}, on any almost complex manifold with $b^1 = 1$ we always have $h^{1,0}_d = 0$. This includes the Inoue surfaces and the secondary Kodaira surface. This leaves only 3 cases (b, c and e in \cite[Table 6.1]{B}) each of which we shall consider in turn. These cases all have $b^1 = 2$ or $3$ and therefore the only possible values for $h^{1,0}_d$ are 0 or 1.

\subsection{Primary Kodaira surface}\label{subsection primary}

The primary Kodaira surface, here denoted by $M$, is given by  $\left(\mathbb Z \times H_3( \mathbb Z)\right) \backslash \left(\mathbb R \times H_3(\mathbb R) \right)$. That is, the quotient of the group $\mathbb R \times H_3(\mathbb R) $ by the (non normal) discrete subgroup  $\mathbb Z \times H_3( \mathbb Z)$ acting by left multiplication. 

Here $H_3(\R)$ denotes the Heisenberg group
\begin{equation*}
H_3(\R)=\left\{
\begin{pmatrix}1&x&z\\0&1&y\\0&0&1\end{pmatrix}
\in\GL(3,\R)\right\},
\end{equation*}
and $H_3(\Z)$ is the subgroup $H_3(\R)\cap\GL(3,\Z)$. Let $t$ denote the coordinate on $\mathbb R$, and $x,y,z$ the coordinates on $H_3(\R)$. The vector fields
\begin{equation*}
e_1=\de{}{t},\ e_2=\de{}{x},\ e_3=\de{}{y}+x\de{}{z},\ e_4=\de{}{z}
\end{equation*}
can be defined on the underlying group and in fact, since they are left invariant, are also well defined on the manifold $M$. They form a basis of $T_pM$ at each point $p\in M$. The dual left invariant coframe is given by
\begin{equation*}
e^1=dt,\ e^2=dx,\  e^3=dy,\ e^4=dz-xdy.
\end{equation*}

Endow $M$ with the left invariant integrable almost complex structure $J$ given by $\phi^1,\phi^2$ being a global coframe of $(1,0)$-forms, where
\[
\phi^1=e^2+ie^3,\ \ \ \phi^2=e^4+ie^1.
\]
 Their structure equations are
\[
\begin{cases}
d\phi^1=0,\\ d\phi^2=\frac{i}2\phi^{1\c1}.
\end{cases}
\]
Their dual frame is given by
\[
V_1=\frac12(e_2-ie_3),\ \ \ V_2=\frac12(e_4-ie_1).
\]
Endow $(M,J)$ with an almost Hermitian metric $\omega$, which can be written as
\begin{equation}\label{eq herm metr prim}
2\omega=is^2\phi^{1\bar{1}}+ir^2\phi^{2\bar{2}}+u\phi^{1\bar{2}}-\overline{u}\phi^{2\bar{1}},
\end{equation}
with $s,r$ real, positive valued smooth functions and $u$ a complex valued smooth function satisfying
\[
rs>|u|.
\]

\begin{lemma}\label{lemma-kodaira}
Let $\sigma=f\phi^1+g\phi^2$ be a $(1,0)$-form on $(M,J)$, where $f,g$ are complex valued smooth functions. Then
\[
d\sigma=0\ \ \  \iff \ \ \ f\in\C,\ g=0.
\]
\end{lemma}
\begin{proof}
It suffices to prove the $\implies$ part. If $d\sigma=0$, then also $\delbar\sigma=0$ and thus
\begin{align*}
0=\delbar\sigma=-\c{V_1}f\phi^{1\c1}-\c{V_2}f\phi^{1\c2}-\c{V_1}g\phi^{2\c1}-\c{V_2}g\phi^{2\c2}+g\frac{i}2\phi^{1\c1},
\end{align*}
and in particular $\c{V_1}g=\c{V_2}g=0$, which implies $V_1\c{V_1}g+V_2\c{V_2}g=0$ and so $g\in\C$ by the maximum principle for second order strongly elliptic operators. Now, since $g\in\C$, it is also the case that $V_1\c{V_1}f+V_2\c{V_2}f=0$ and so $f\in\C$, which finally yields $g=0$.
\end{proof}

\begin{proposition}\label{prop-prim-kod}
Let $(M,J)$ be endowed with the almost Hermitian metric $\omega$ defined as in \eqref{eq herm metr prim}.
\begin{enumerate}
\item[I)] If $-\frac{i}2\c{V_2}\c{u}+s\c{V_1}s\neq 0$, then 
$$
\mathcal{H}^{1,0}_d=\{0\},\qquad 
\mathcal{H}^{1,0}_{\delbar}=\C < \Phi^1>.
$$
\item[II)] If $-\frac{i}2\c{V_2}\c{u}+s\c{V_1}s= 0$, then 
$$
\mathcal{H}^{1,0}_d=\mathcal{H}^{1,0}_{\delbar}=\C < \Phi^1> .
$$
\end{enumerate}
\end{proposition}
\begin{proof}
Take $\sigma=f\phi^1$ to be a general $d$-closed $(1,0)$-form on $(M,J)$, with $f\in\C$, by Lemma \ref{lemma-kodaira}. Its Hodge $*$ is
$$
*\sigma =\frac{i}{2}f\overline{u}\phi^{12\c{1}}+
\frac{1}{2}fs^2\phi^{12\c{2}}.
$$
We compute 
\[
d*\sigma=f\left(-\frac{i}2\c{V_2}\c{u}+s\c{V_1}s\right)\phi^{12\c1\c2},
\]
which ends the proof.
\end{proof}

In particular, if $u$ and $s$ are constant we have $h^{1,0}_d = 1$. Alternatively, we can choose $u$ to be constant and
$
s=e^{F(t,x)},
$
where $F\in C^{\infty}(M)$ is a smooth function on $M$ depending only on the coordinates $t$ and $x$, \textit{i.e.} a function on $\mathbb{R}^2$ satisfying
$$F(t+t_0,x+x_0)=F(t,x)
$$ 
for any given $(t_0,x_0)\in \mathbb{Z}^2$. In this case, $h^{1,0}_d=0$ iff
$$
\overline{V}_1(F)\neq0.
$$
This is sufficient to prove Theorem \ref{thm-10d} (note that the almost complex structure in this example is integrable).

\subsection{Compact quotient of Nil$^4$}
We start by recalling the definition of the group $\hbox{\rm Nil}^4$ as $(\mathbb{R}^4, *)$ where the product is defined as 
$$(a, b, c, d) * (x, y, z, t) = (x + a, y + b, z + ay + c, t + \frac{1}{2} a^2 y + az + d),
$$
for any $(a, b, c, d), (x, y, z, t)\in \R^4$. 

A straightforward calculation shows that the set of $1$-forms
\begin{equation}
\left\{
\begin{array}{l}
E^1 = dx,\\
E^2 = dy,\\
E^3 = dz - x dy,\\
E^4 = dt + \frac{1}{2} x^2 dy - x dz
\end{array}
\right.
\end{equation}
on ${\rm Nil}^4$ is a left invariant coframe. The dual frame is given by 
\begin{equation}\label{frame-Nil4}
\left\{
\begin{array}{l}
E_1 = \frac{\partial}{\partial x},\\
E_2 = \frac{\partial}{\partial y} + x \frac{\partial}{\partial z} + \frac{1}{2} x^2 \frac{\partial}{\partial t},\\
E_3 = \frac{\partial}{\partial z} + x \frac{\partial}{\partial t},\\
E_4 = \frac{\partial}{\partial t}.
\end{array}
\right.
\end{equation}
Then, the non trivial commutators are given by
$$
[E_1, E_2] = E_3, \qquad [E_1, E_3] = E_4,
$$
showing that $\hbox{\rm Nil}^4$ is a simply-connected nilpotent Lie group.
It turns out that $\Gamma =2\mathbb{Z}^4$ is a discrete uniform subgroup of $\hbox{\rm Nil}^4$, and thus $\Gamma\backslash \hbox{\rm Nil}^4$ is a compact $4$-nilmanifold; we set $M_{Nil}=\Gamma\backslash \hbox{\rm Nil}^4$. Then, the coframe $\{E^1,\ldots,E^4\}$ satisfies the following structure equations on $M_{Nil}$
\begin{equation}\label{structure-equations-Nil4}
dE^1 = 0, \qquad dE^2 = 0, \qquad dE^3 = -E^1 \wedge E^2, \qquad dE^4 = -E^1 \wedge E^3.
\end{equation}
It is known that $M_{Nil}$ has no complex structure.
Let us define an almost complex structure $J$ on $M$ by defining a global coframe of $(1,0)$-forms as
$$
\Phi^1=E^1+iE^2,\qquad \Phi^2=E^3+iE^4.
$$
Then, by \eqref{structure-equations-Nil4}, we obtain the complex structure equations 
\begin{equation}\label{complex-structure-equations-Nil4}
d\Phi^1=0,\qquad
d\Phi^2=-\frac{i}{4}\left(\Phi^{12}+2\Phi^{1\bar{1}}+\Phi^{1\bar{2}}-\Phi^{2\bar{1}}+\Phi^{\bar{1}\bar{2}}\right).
\end{equation}
 Let $g$ be the almost Hermitian metric on $M_{Nil}$ whose fundamental form is given by
\begin{equation}\label{metric-Nil4}
2\omega=ir^2\Phi^{1\bar{1}}+is^2\Phi^{2\bar{2}}+u\Phi^{1\bar{2}}-\overline{u}\Phi^{2\bar{1}},
\end{equation}
where $r,s$, respectively $u$, are smooth, positive real valued functions, respectively complex valued function on $M_{Nil}$, satisfying
$$
rs>\vert u\vert.
$$
\begin{lemma}
Let $\sigma=f\Phi^1+g\Phi^2\in A^{1,0}(M_{Nil})$, where $f$, $g$ are smooth complex valued functions on $M_{Nil}$. Then  
$$
d\sigma=0\ \ \  \iff\ \ \   f\in\C,\  g=0.
$$ 
\end{lemma}
\begin{proof}
It suffices to prove the $\implies$ part. If $d\sigma=0$, then also $\delbar\sigma=0$ and thus
\begin{align*}
0=\delbar\sigma=& -\overline{V}_1(f)\Phi^{1\bar{1}}-\overline{V}_2(f)\Phi^{1\bar{2}}- \overline{V}_1(g)\Phi^{2\bar{1}}-\overline{V}_2(g)\Phi^{2\bar{2}}+\\
& -\frac{i}{2}g\Phi^{1\bar{1}}-\frac{i}{4}g\Phi^{1\bar{2}}
+\frac{i}{4} g\Phi^{2\bar{1}}.
\end{align*}
By the above calculation,
$\delbar\sigma =0$ if and only if 
\begin{equation*}
\begin{cases}
\overline{V}_1(f)+\frac{i}{2} g=0,\\
\overline{V}_2(f)+\frac{i}{4} g=0,\\
\overline{V}_1(g)-\frac{i}{4} g=0,\\
\overline{V}_2(g)=0.
\end{cases}
\end{equation*}
Note that the projection $\pi:\R^4\to\R^2$ defined by $\pi(x,y,z,t)=(x,y)$ gives rise to a well defined map $\pi: M_{Nil}\to \mathbb{T}^2$. Therefore, the last equation gives immediately that  
$$
V_2\overline{V}_2(g)=0,
$$
and since
$$
V_2\overline{V}_2=\left(\partial^2_{zz}+2x\partial^2_{tz}+(1+x^2)\partial^2_{tt}\right),
$$
for any fixed $(x,y)$ we can apply the maximum principle for second order strongly elliptic operators on the compact submanifold $\pi^{-1}(x,y)$, implying
$$
g=g(x,y).
$$
Now from the second equation we get
$$
V_2\overline{V}_2(f)=0,
$$
which, arguing as before, yields
$$
f=f(x,y).
$$
Therefore, from the second equation we derive $g=0$, and consequently the first equation implies that $f\in \C$. The Lemma is proved.
\end{proof}

Arguing as in Proposition \ref{prop-prim-kod}, we obtain
\begin{proposition}
Let $(M_{Nil},J)$ be endowed with the almost Hermitian metric $\omega$ defined as in \eqref{metric-Nil4}. 
\begin{enumerate}
\item[I)] If $-\frac{i}2\c{V_2}\c{u}+s\c{V_1}s\neq 0$, then 
$$
\mathcal{H}^{1,0}_d=\{0\},\qquad 
\mathcal{H}^{1,0}_{\delbar}=\C < \Phi^1>.
$$
\item[II)] If $-\frac{i}2\c{V_2}\c{u}+s\c{V_1}s= 0$, then 
$$
\mathcal{H}^{1,0}_d=\mathcal{H}^{1,0}_{\delbar}=\C < \Phi^1> .
$$
\end{enumerate}
\end{proposition}

In particular, we can choose $u$ to be constant and
$
s=e^{F(x,y)},
$
where $F \in C^{\infty}(M_{Nil})$ is a smooth function on $M_{Nil}$ depending only on the coordinates $x$ and $y$, \textit{i.e.} a function on $\mathbb{R}^2$ satisfying
$$F(x+2\alpha,y+2\beta)=F(x,y)
$$ 
for any given $(\alpha,\beta)\in \mathbb{Z}^2$. In this case, $h^{1,0}_d=0$ iff
$$
\overline{V}_1(F)\neq0.
$$

\subsection{Compact quotient of $\mathbb{R} \times $Sol$^3$}

Let $\Sol$ denote the group $(\R^4,*)$, where the product is defined as 
 $$(s,a,b,c) * (t,x,y,z) = (s+t, a+x, b + e^s y, c+ e^{-s} z)$$
 for any $(s,a,b,c), (t,x,y,z) \in \mathbb{R}^4$.
The set of 1-forms
$$\begin{cases}
    E^1 = dt,\\
    E^2 = dx,\\
    E^3 = e^{-t}dy,  \\
    E^4 = e^{t} dz
\end{cases}$$
is a left invariant coframe on $\Sol$, and the dual frame is given by
$$\begin{cases}
    E_1 = \frac{\partial}{\partial t},\\
    E_2 = \frac{\partial}{\partial x},\\
    E_3 = e^{t}\frac{\partial}{\partial y}, \\
    E_4 = e^{-t}\frac{\partial}{\partial z}.
\end{cases}$$
The above coframe satisfies the structure equations
\begin{equation}\label{str eq sol3}
    dE^1 = 0, \quad \quad dE^2 = 0, \quad \quad dE^3 = -E^{1}\wedge E^3, \quad \quad dE^4 = E^{1}\wedge E^4.
\end{equation}
Let us define a left invariant almost complex structure $J$ on $\Sol$ by defining a global coframe of $(1,0)$-forms 
$$\Phi^1 = E^1 + i E^2, \quad \quad \quad \quad \Phi^2 =  E^3 + i E^4. $$
Then, by \eqref{str eq sol3}, we obtain the structure equations
$$
\begin{cases}
d\Phi^1 = 0,\\ 
d\Phi^2 = \frac 12 \left(\Phi^{1 \bar 2} + \Phi^{\bar 1 \bar 2}\right).
\end{cases}
$$

Given any integer-valued matrix $A \in SL_2(\mathbb Z)$ such that $\tr A \geq 3$, we can define a discrete uniform subgroup $\Gamma_A$ of $\Sol$ generated by
$$\Gamma_A = \left< (\kappa,0,0,0), (0,1,0,0), (0,0,p,q), (0,0,r,s) \right>.  $$
Here $\kappa, p, q, r, s \in \mathbb{R}, \kappa > 0$ are defined such that
$$ A 
 = \begin{pmatrix} p & r \\ q & s\end{pmatrix}^{-1}
\begin{pmatrix} e^\kappa & 0 \\ 0 & e^{-\kappa}\end{pmatrix}
\begin{pmatrix} p & r \\ q & s\end{pmatrix}.$$
We thereby obtain a solvmanifold $M_A = \Gamma_A \backslash (\Sol).$

Let $g$ be the almost Hermitian metric on $M_A$ whose fundamental form is given by 
\begin{equation}\label{herm metric sol3}
    2\omega = i r^2 \Phi^{1 \bar 1} + is^2 \Phi^{2 \bar 2} + u \Phi^{1\bar 2} - \overline u \Phi^{2\bar 1} 
\end{equation}
where $r$ and $s$ are smooth, positive real valued functions and $u$ is a smooth, complex valued function satisfying
$$rs > |u|.$$

\begin{lemma}
Let $\sigma = f\Phi^1 + g \Phi^2 \in A^{1,0}(M_A)$, where $f$ and $g$ are smooth, complex valued functions on $M_A$. Then
$$d\sigma = 0\ \ \  \iff  f \in \mathbb{C},\ g=0. $$
\end{lemma}
\begin{proof}
 It suffices to prove the $\implies$ part. If $d\sigma=0$, then also $\delbar\sigma=0$ and thus
    $$0=\overline \partial \sigma = -\overline{V}_1 (f)\Phi^{1\bar{1}}-\overline{V}_2(f)\Phi^{1\bar{2}}- \overline{V}_1(g)\Phi^{2\bar{1}}-\overline{V}_2(g)\Phi^{2\bar{2}}+ \frac 12 g \Phi^{1 \bar 2}. $$
We conclude arguing as in Lemma \ref{lemma-kodaira}.
\end{proof}

Again, as in Proposition \ref{prop-prim-kod} we conclude
\begin{proposition}
Let $(M_A, J)$ be endowed with the almost Hermitian metric $\omega$ defined above in \eqref{herm metric sol3}.
\begin{enumerate}
\item[I)] If $-\frac{i}2\c{V_2}\c{u}+s\c{V_1}s\neq 0$, then 
$$
\mathcal{H}^{1,0}_d=\{0\},\qquad 
\mathcal{H}^{1,0}_{\delbar}=\C < \Phi^1>.
$$
\item[II)] If $-\frac{i}2\c{V_2}\c{u}+s\c{V_1}s= 0$, then 
$$
\mathcal{H}^{1,0}_d=\mathcal{H}^{1,0}_{\delbar}=\C < \Phi^1> .
$$
\end{enumerate}
\end{proposition}

\section{$h^{p,0}_{d}$ on compact almost Hermitian manifolds}\label{sec-invariants-p0}

In this section we are interested in studying the properties of the numbers $h^{p,0}_d$ in the higher dimensional case. The following Lemma shows that $h^{p,0}_d$ does not vary with different choices of almost K\"ahler metrics.

\begin{lemma}\label{lemma-p0 kahler harm}
Let $(M,J,g,\omega)$ be a compact almost K\"ahler manifold of real dimension $2n$. For any bidegree $(p,q)\in\{(p',0),(0,q'),(p',n),(n,q')\,:\,0\le p',q'\le n\}$, the injection $j$ of Lemma \ref{lemma-injection} is an isomorphism, $h^{p,q}_d= h^{p,q}_J$,
\[
\H^{p,q}_d=\ker d\cap A^{p,q},
\]
and there are no almost Hermitian metrics for which $h^{p,q}_d$ takes a larger value.
\end{lemma}
\begin{proof}
All $(p,0)$-forms are primitive and so $\H^{p,0}_d=\ker d\cap A^{p,0}$ by \eqref{eq-*-primitive}. This implies that $j$ is also surjective. The maximality of $h^{p,0}_d$ follows by the inequality of Lemma \ref{lemma-injection}. The other bidegrees follow by \eqref{eq numbers}.
\end{proof}

The following remark gives a possible hint as to how it could be possible to generalise the previous result for any bidegree.
\begin{remark}
Let $(M,J,g,\omega)$ be a $2n$-dimensional compact almost K\"ahler manifold. Cirici and Wilson \cite[Corollary 5.4]{CW} proved the following primitive decomposition:
\[
\H^{p,q}_d=\bigoplus_{r\geq\max(p+q-n,0)}L^r\left(\H^{p-r,q-r}_d\cap P^{p-r,q-r}\right).
\]
By \eqref{eq-*-primitive}, note that  for any bidegree $(p,q)$ we have
\[
\H^{p,q}_d\cap P^{p,q}=\ker d \cap P^{p,q}.
\]
However, since the space $\ker d \cap P^{p,q}$ in general depends on the choice of the almost K\"ahler metric $\omega$, we cannot directly conclude that $h^{p,q}_d$ is almost K\"ahler invariant.
\end{remark}

We end the section by proving a generalisation of \cite[Lemma 2.12]{DLZ} in higher dimension.

\begin{theorem}\label{n0}
Let $M$ be a $2n$-dimensional compact manifold endowed with a non integrable almost complex structure $J$. Then the only $d$-closed $(n,0)$-form is the zero form. This is equivalent to $h^{n,0}_d=0$ and to $h^{n,0}_J=0$.
\end{theorem}
\begin{proof}
Let us first show that if $J$ is non integrable then the only $d$-closed $(n,0)$-form is the zero form.
Let $\psi\in A^{n,0}$ be a $d$-closed form and fix any almost Hermitian metric. Since $\psi$ is primitive, then by \eqref{eq-*-primitive} the closed form $\psi$ is also $d$-harmonic. For any point $p\in M$, fix $\phi^1,\dots,\phi^n$ as a local coframe of $(1,0)$-forms centered on $p$. Then $\psi$ can be locally written as
\[
\psi=f\phi^1\wedge\dots\wedge\phi^n,
\]
where $f$ is a smooth function with complex values. We compute
\begin{align*}
0=\mubar\psi=&f\mubar(\phi^1\wedge\dots\wedge\phi^n)\\
=&f\big(\mubar\phi^1\wedge\phi^2\dots\wedge\phi^n-\phi^1\wedge\mubar\phi^2\wedge\dots\wedge\phi^n+\dots\\
+&(-1)^{n-1}\phi^1\wedge\dots\wedge\phi^{n-1}\wedge\mubar\phi^n\big).
\end{align*}
Since $\mubar\phi^j\in A^{0,2}$, it follows that either $f(p)=0$ (and so $\psi(p)=0$) or $\mubar(p)=0$.
 
In order to prove the claim by contradiction, assume that there exists a $d$-closed form $\psi\not\equiv0$. 
Since $\psi\not\equiv0$ and it is $d$-harmonic, then the set
\[
\{p\in M\,|\,\psi(p)\ne0\}
\]
is dense in $M$ by the unique continuation theorem for harmonic differential forms \cite{AKS} (if $\psi\equiv0$ on an open set, then $\psi\equiv0$ everywhere on $M$) applied to the real and imaginary parts of $\psi$.
Therefore, the set
\[
\{p\in M\,|\,\mubar(p)=0\}\supseteq\{p\in M\,|\,\psi(p)\ne0\}
\]
is also dense in $M$. By continuity, this means that $\mubar\equiv0$ everywhere on $M$, i.e., $J$ is integrable, which is a contradiction.

Finally, let us show that $h^{n,0}_d=0$ iff $h^{n,0}_J=0$ iff the only $d$-closed $(n,0)$-form is the zero form.
Clearly, if the only $d$-closed $(n,0)$-form is the zero form, then $h^{n,0}_d=0$ and $h^{n,0}_J=0$.
On the other hand, by \eqref{eq-*-primitive}, if $h^{n,0}_d=0$ then the only $d$-closed $(n,0)$-form is the zero form.
Moreover, if $h^{n,0}_J=0$, then the only $d$-closed $(n,0)$-forms are $d$-exact, but by \eqref{eq-*-primitive} they would be also $d$-harmonic, which cannot be since exact and harmonic forms are $L^2$ orthogonal.
\end{proof}

\section{Harmonic $(1,1)$-forms on compact quotients of Lie groups}
\label{sec quotient lie 11}

In this section we will focus on harmonic forms of bidegree $(1,1)$. The next Lemma gives a characterisation of harmonic $(1,1)$-forms in terms of their primitive decomposition on compact almost K\"ahler manifolds.

\begin{lemma}\label{lemma-primitive-11}
Let $(M,J,g,\omega)$ be a compact almost K\"ahler manifold of dimension $2n$. Let $\alpha$ be a $(1,1)$-form: its primitive decomposition is
\[
\alpha=\beta+f\omega,
\]
with $f\in\cinf(M;\C)$ and $\beta\in P^{1,1}$, i.e., $\beta\in A^{1,1}$ and $\omega^{n-1}\wedge\beta=0$. Assume that $\alpha$ is closed. Then $\alpha$ is harmonic iff $f$ is constant.
\end{lemma}
\begin{proof}
The Hodge $*$ operator of $\alpha$ is computed via \eqref{eq-*-primitive}:
\[
*\alpha=-\frac{1}{(n-2)!}\omega^{n-2}\wedge\beta+\frac{f}{(n-1)!}\omega^{n-1}.
\]
Furthermore, the form $\alpha$ is closed iff
\[
d\beta+df\wedge\omega=0,
\]
and it is co-closed iff
\[
-\frac{1}{(n-2)!}\omega^{n-2}\wedge d\beta+\frac{1}{(n-1)!}df\wedge\omega^{n-1}=0.
\]
Assuming that $\alpha$ is closed, then $\alpha$ is co-closed iff
\[
\left(\frac{1}{(n-2)!}+\frac{1}{(n-1)!}\right)df\wedge\omega^{n-1}=0,
\]
iff $df=0$ by the Lefschetz isomorphism and so iff $f$ is constant.
\end{proof}

Focusing on the special case of compact quotients of Lie groups endowed with almost K\"ahler structures, we obtain the following result which, as we will see, has implications both on the cohomology spaces $H^{1,1}_J$ and on the space of harmonic $(1,1)$-forms $\H^{1,1}_d$.

\begin{theorem}\label{thm kahler harmonic inv}
Let $G$ be a $2n$-dimensional Lie group endowed with an almost K\"ahler structure $(J,g)$ such that the fundamental form $\omega$ is left invariant, and let $\Gamma$ be a discrete subgroup such that $M=\Gamma\backslash G$ is compact. Then every left invariant closed $(1,1)$-form on $M$ is harmonic.
\end{theorem}
\begin{proof}
Let $\alpha$ be a left invariant closed $(1,1)$-form on $M$. Its primitive decomposition is
\[
\alpha=\beta+f\omega,
\]
with $\beta\in P^{1,1}$, i.e., $\beta\in A^{1,1}$ and $\omega^{n-1}\wedge\beta=0$, and $f\in\cinf(M)$. Observe that
\[
\omega^{n-1}\wedge\alpha=f\omega^{n}
\]
and it is a left invariant form, therefore $f$ is a complex constant. Then, since $\alpha$ and $\omega$ are closed, we can deduce
\[
d\beta=0.
\]
Since $\beta$ is primitive, it follows that $\beta$ is also harmonic by \eqref{eq-*-primitive}. Considering that also $\omega$ is harmonic, we conclude that $\alpha$ is harmonic.
\end{proof}

If, in addition to the assumptions of the previous theorem, we also require that the $(1,1)$-representatives of the cohomology space $H^{1,1}_J$ are left invariant, we obtain that $h^{1,1}_d$ does not depend on the choice of almost K\"ahler metrics with left invariant fundamental form.

\begin{corollary}\label{cor kahler inv cohom}
Let $G$ be a $2n$-dimensional Lie group endowed with an almost K\"ahler structure $(J,g)$ such that the fundamental form $\omega$ is left invariant, and let $\Gamma$ be a discrete subgroup such that $M=\Gamma\backslash G$ is compact. Assume that
\[
H^{1,1}_J(M)=\Span_\C\langle[\phi_1]_{dR},\dots,[\phi_k]_{dR}\rangle,
\]
with $\phi_j$ left invariant closed $(1,1)$-forms. Then $h^{1,1}_d=h^{1,1}_J$ and
\[
\H^{1,1}_d=\Span_\C\langle\phi_1,\dots,\phi_k\rangle.
\]
In particular $h^{1,1}_d$ does not depend on the choice of almost K\"ahler metrics with left invariant fundamental form, and there are no almost Hermitian metrics for which $h^{1,1}_d$ takes a larger value.
\end{corollary}
\begin{proof}
By Lemma \ref{lemma-injection}, we always have $h^{1,1}_d\le h^{1,1}_J$, and since $\phi_1,\dots,\phi_k$ are harmonic by Theorem \ref{thm kahler harmonic inv}, we obtain the equality. The maximality of $h^{1,1}_d$ follows by the inequality of Lemma \ref{lemma-injection}.
\end{proof}

Note that the requirement that the $(1,1)$-representatives of the cohomology space $H^{1,1}_J$ are left invariant and that the fundamental form $\omega$ is left invariant does not force $J$ to be left invariant.

The following proposition tells us how harmonic forms and the cohomology space $H^{p,q}_J$ behave when both the almost Hermitian structure $(J,g,\omega)$ and the de Rham cohomology $H^k_{dR}$ are left invariant.

\begin{proposition}\label{prop-isom-cohom-inv}
Let $G$ be a $2n$-dimensional Lie group endowed with a left invariant almost Hermitian structure $(J,g,\omega)$, and let $\Gamma$ be a discrete subgroup such that $M=\Gamma\backslash G$ is compact. Assume that for some degree $k$ the de Rham cohomology of $M$ is isomorphic to the cohomology of the Lie algebra $\mathfrak{g}$ of $G$, i.e.,
\[
H^k_{dR}(M)\cong H^k(\mathfrak{g}).
\]
Then every $k$-harmonic form on $M$ is left invariant and there is an isomorphism
\[
H^{p,q}_J(M)\cong H^{p,q}_J(\mathfrak{g})
\]
for any $p+q=k$.
\end{proposition}
\begin{proof}
For the first claim, note that by the Hodge isomorphism on the Lie algebra $\mathfrak{g}$ we know that the cohomology space $H^k(\mathfrak{g})$ is isomorphic to the space of harmonic $k$-forms on $\mathfrak{g}$. Moreover, since $H^k(\mathfrak{g})\cong H^k_{dR}(M)\cong \H^k_d$, we derive that all possible harmonic $k$-forms on $M$ are given by harmonic $k$-forms on $\mathfrak{g}$, therefore they are left invariant.

For the second claim, we recall some of the properties of symmetrisation operator
\[
\nu:A^k(M)\to \bigwedge^k \mathfrak{g}^*,
\]
see, e.g., \cite[p. 192]{ugarte} or \cite{P}.
It is linear and surjective, it commutes with the exterior derivative and with the left invariant almost complex structure, namely $d(\nu(\alpha))=\nu(d(\alpha))$ and $J(\nu(\alpha))=\nu(J(\alpha))$. Therefore it induces a linear surjection in cohomology
\begin{align*}
\tilde{\nu}:&\,H^k_{dR}(M)\to H^k(\mathfrak{g}),\\
&\,[\alpha]\mapsto[\nu(\alpha)].
\end{align*}
Since we are assuming that these two spaces are isomorphic, it follows that $\tilde{\nu}$ is an isomorphism. It is then easy to verify that $\tilde{\nu}$ induces also an isomorphism 
\[
\tilde{\nu}:H^{p,q}_J(M)\to H^{p,q}_J(\mathfrak{g}).\qedhere
\]
\end{proof}

Hattori's Theorem \cite{Ha} says that on a compact quotient of a completely solvable Lie group by a discrete subgroup the de Rham cohomology is isomorphic to the cohomology of the Lie algebra. Therefore from Corollary \ref{cor kahler inv cohom} and Proposition \ref{prop-isom-cohom-inv} we obtain the following result.

\begin{corollary}\label{cor-nilm}
Let $G$ be a $2n$-dimensional completely solvable Lie group endowed with a left invariant almost K\"ahler structure $(J,g,\omega)$, and let $\Gamma$ be a discrete subgroup such that $M=\Gamma\backslash G$ is compact.
Then $h^{1,1}_d=h^{1,1}_J$ does not depend on the choice of the left invariant almost K\"ahler metric, and there are no almost Hermitian metrics for which $h^{1,1}_d$ takes a larger value. If $2n=4$, then $h^{1,1}_J=b^-+1$.
\end{corollary}
\begin{proof}
The first claim follows directly from the above mentioned results. If $2n=4$, recall that $h^{1,1}_d=b^-+1$ by the proof of \cite[Proposition 6.1]{HZ}.
\end{proof}

\begin{corollary}\label{cor-ineq-serre}
Let $G$ be a $2n$-dimensional completely solvable Lie group endowed with a left invariant almost K\"ahler structure $(J,g,\omega)$, and let $\Gamma$ be a discrete subgroup such that $M=\Gamma\backslash G$ is compact.
Then $h^{n-1,n-1}_d$ does not depend on the choice of the left invariant almost K\"ahler metric, and there are no almost Hermitian metrics for which $h^{n-1,n-1}_d$ takes a larger value.
Furthermore
\[
h^{1,1}_J\le h^{n-1,n-1}_J.
\]
\end{corollary}
\begin{proof}
The first statement follows since $h^{1,1}_d=h^{n-1,n-1}_d$ by \eqref{eq numbers}. The second statement follows from Lemma \ref{lemma-injection} since $h^{1,1}_J=h^{1,1}_d=h^{n-1,n-1}_d\le h^{n-1,n-1}_J$.
\end{proof}

We now focus on dimension 4, weakening the almost K\"ahler assumption and studying the dimension of $h^{1,1}_J$. Recall that we always have the lower bound $ h^{1,1}_J \ge b^-$, since harmonic anti self dual forms provide linearly independent classes in $H^{1,1}_J$. By \cite[Theorem 3.3]{DLZ}, we also have the upper bound $h^{1,1}_J\le b^-+1$ if $J$ is tamed by a symplectic form and so in particular if the manifold admits a compatible almost K\"ahler metric. Without any assumption of almost K\"ahlerness, we recall the following results of \cite{PT4,Ho}.

\begin{theorem}[{\cite[Theorem 4.3]{PT4}}]\label{lem 4 gaud bc cost}
Let $(M,J,g,\omega)$ be a compact almost Hermitian manifold of real dimension $4$. Assume that $g$ is Gauduchon, then
\begin{equation*}
\H^{1,1}_{BC}=\{f\omega+\gamma\in A^{1,1}\,|\,f\in\C,\ *\gamma=-\gamma,\ d(f\omega+\gamma)=0\}.
\end{equation*}
\end{theorem}
\begin{theorem}[{\cite[Theorem 4.2]{Ho}}]\label{thm holt bc}
Let $(M,J,g,\omega)$ be a compact almost Hermitian manifold of real dimension $4$. Then $h^{1,1}_{BC}=b^-+1$.
\end{theorem}

In the special setting of compact quotients of 4-dimensional Lie groups with left invariant $\del\delbar$-closed fundamental form and left invariant $(1,1)$-cohomology, we obtain an injection of $H^{1,1}_J$ into the space of Bott-Chern harmonic $(1,1)$-forms.

\begin{theorem}\label{thm-gaud-bc-inj}
Let $G$ be a $4$-dimensional Lie group endowed with an almost complex structure $J$ and a Gauduchon metric $g$ such that the fundamental form $\omega$ is left invariant. Let $\Gamma$ be a discrete subgroup such that $M=\Gamma\backslash G$ is compact. Assume that
\[
H^{1,1}_J(M)=\Span_\C\langle[\phi_1]_{dR},\dots,[\phi_k]_{dR}\rangle,
\]
with $\phi_1,\dots,\phi_k$ left invariant closed $(1,1)$-forms.
Then $h^{1,1}_J\le h^{1,1}_{BC}=b^-+1$ and
\begin{equation}\label{eq h11bc inv incl}
\Span_\C\langle\phi_1,\dots,\phi_k\rangle\subseteq\H^{1,1}_{BC}(M,J,g,\omega).
\end{equation}
\end{theorem}
\begin{proof}
The primitive decomposition of every $\phi_j$ is
\[
\phi_j=f_j\omega+\gamma_j,
\]
with $f_j\in\C$ (since $\omega$ is left invariant) and $*\gamma_j=-\gamma_j$. By Theorem \ref{lem 4 gaud bc cost}, we obtain that $\phi_j\in\H^{1,1}_{BC}$. Moreover, since $\{[\phi_j]_{dR}\}_{j=1,\dots,k}$ are linearly independent de Rham cohomology classes, in particular $\{\phi_j\}_{j=1,\dots,k}$ are linearly independent forms. Therefore we have a linear injection given by
\begin{align*}
&H^{1,1}_J\hookrightarrow \H^{1,1}_{BC}\\
&[\phi_j]_{dR}\mapsto \phi_j\ \forall j=1,\dots,k
\end{align*}
which proves the theorem.
\end{proof}

Under the same assumptions, we are able to characterise when $h^{1,1}_J$ equals $b^-$ or $b^-+1$. The condition implicitly involves the space of Bott-Chern harmonic $(1,1)$-forms. Note that by Theorems \ref{lem 4 gaud bc cost} and \ref{thm holt bc} we know that there always exists a closed form $c\omega+\gamma$ with $0\ne c\in\C$ and $*\gamma=-\gamma$.

\begin{theorem}
Let $G$ be a $4$-dimensional Lie group endowed with an almost complex structure $J$ and a Gauduchon metric $g$ such that the fundamental form $\omega$ is left invariant. Let $\Gamma$ be a discrete subgroup such that $M=\Gamma\backslash G$ is compact. Assume that
\[
H^{1,1}_J(M)=\Span_\C\langle[\phi_1]_{dR},\dots,[\phi_k]_{dR}\rangle,
\]
with $\phi_1,\dots,\phi_k$ left invariant closed $(1,1)$-forms. Given any closed form $c\omega+\gamma$ with $0\ne c\in\C$ and $*\gamma=-\gamma$,
then $h^{1,1}_J=b^-$ iff $c\omega+\gamma$ is $d$-exact; otherwise $h^{1,1}_J=b^-+1$.
\end{theorem}
\begin{proof}
By Theorems \ref{lem 4 gaud bc cost} and \ref{thm holt bc} we know $h^{1,1}_{BC}=b^-+1$ and
\begin{equation}\label{eq char bc harm 11}
\H^{1,1}_{BC}=\Span_\C<c\omega+\gamma,\gamma_1,\dots,\gamma_{b-}>,
\end{equation}
where $d\gamma_j=0$ and $*\gamma_j=-\gamma_j$. Then, by \eqref{eq h11bc inv incl}, we obtain that $h^{1,1}_J=b^-+1$ iff the de Rham cohomology classes $[c\omega+\gamma]_{dR},[\gamma_1]_{dR},\dots,[\gamma_{b-}]_{dR}$ are linearly independent. Assume that a linear combination of the representatives is exact, i.e.,
\begin{equation}\label{eq linear comb}
A(c\omega+\gamma)+\sum A_j\gamma_j=d\eta,
\end{equation}
with $A,A_j\in\C$, then $d\eta-A(c\omega+\gamma)$ is a primitive closed form. If the primitive decomposition of $d\eta$ is
\[
d\eta=f\omega+\tilde\gamma,
\]
with $f\in\cinf(M;\C)$ and $*\tilde\gamma=-\tilde\gamma$, then it must be $f\equiv Ac$. By Theorem \ref{lem 4 gaud bc cost}, we obtain that
 $d\eta\in\H^{1,1}_{BC}$. Since every $\gamma_j$ is harmonic, and harmonic and exact forms are $L^2$ orthogonal, from \eqref{eq char bc harm 11} it follows that $d\eta$ is a multiple of $c\omega+\gamma$ and since $f\equiv Ac$ we get $\tilde\gamma=A\gamma$ and so
\[
d\eta=A(c\omega+\gamma).
\]
Now, from \eqref{eq linear comb} we deduce that $A_j=0$ for all $j$; furthermore, we obtain that $A=0$ if $c\omega+\gamma$ is not exact. Therefore, if $c\omega+\gamma$ is not exact, then the cohomology classes $[c\omega+\gamma]_{dR},[\gamma_1]_{dR},\dots,[\gamma_{b-}]_{dR}$ are linearly independent and we obtain $h^{1,1}_J=b^-+1$. On the other hand, if $c\omega+\gamma$ is exact, then the same cohomology classes are linearly dependent and we get $h^{1,1}_J=b^-$.
\end{proof}

Recall that if the almost Hermitian structure $(J,g,\omega)$ is left invariant, then the metric is Gauduchon (see, e.g., \cite[Proposition 3.10]{FTT}), therefore we immediately obtain the following corollary.

\begin{corollary}
Let $G$ be a $4$-dimensional Lie group endowed with a left invariant almost Hermitian structure $(J,g,\omega)$, and let $\Gamma$ be a discrete subgroup such that $M=\Gamma\backslash G$ is compact. Assume that
\[
H^{1,1}_J(M)=\Span_\C\langle[\phi_1]_{dR},\dots,[\phi_k]_{dR}\rangle,
\]
with $\phi_1,\dots,\phi_k$ left invariant closed $(1,1)$-forms. Then $h^{1,1}_J\le b^-+1$ and
\[
\Span_\C\langle\phi_1,\dots,\phi_k\rangle\subseteq\H^{1,1}_{BC}(M,J,g,\omega).
\]
Moreover, given any closed form $c\omega+\gamma$ with $0\ne c\in\C$ and $*\gamma=-\gamma$,
then $h^{1,1}_J=b^-$ iff $c\omega+\gamma$ is $d$-exact; otherwise $h^{1,1}_J=b^-+1$.
\end{corollary}

Moreover, if the 4-dimensional group is completely solvable, we summarise the previous results in the following corollary.

\begin{corollary}\label{cor-compl-solv-non-kahl}
Let $G$ be a $4$-dimensional completely solvable Lie group endowed with a left invariant almost Hermitian structure $(J,g,\omega)$, and let $\Gamma$ be a discrete subgroup such that $M=\Gamma\backslash G$ is compact. Then $h^{1,1}_J\le b^-+1$, there is an injection
\[
H^{1,1}_J(M)\hookrightarrow \H^{1,1}_{BC}(M,J,g,\omega)
\]
and, given any closed form $c\omega+\gamma$ with $0\ne c\in\C$ and $*\gamma=-\gamma$,
we have $h^{1,1}_J=b^-$ iff $c\omega+\gamma$ is $d$-exact; otherwise $h^{1,1}_J=b^-+1$.
\end{corollary}

\section{Examples}\label{sec examples}

In this section we illustrate two examples with a notable relationship between the numbers $h^{p,q}_d$ and $h^{p,q}_J$.

\subsection{Harmonic $(p,q)$-forms on a six dimensional nilmanifold}\label{ex fino tomassini}
Here we provide an example where the inclusion $j$ of Lemma \ref{lemma-injection} is not an isomorphism, i.e., where
\begin{equation}\label{eq-inequality-hodge}
h^{p,q}_J>h^{p,q}_d.
\end{equation}
To be able to make computations, we will work on a 6-dimensional compact nilmanifold. In this case, by Lemma \ref{lemma-injection} and Corollary \ref{cor-nilm}, the strict inequality \eqref{eq-inequality-hodge} is only possible for $(p,q)\in\{(2,1),(1,2)\}$.

Let us consider the 6-dimensional nilmanifold $M$ of \cite[Example 3.3]{FT} endowed with a different almost complex structure. Let $M$ be given by a compact quotient of the 6-dimensional real nilpotent Lie group $G$ with structure equations
\[
\begin{cases}
de^j=0,\ j=1,2,3,4,\\
de^5=e^{12},\\
de^6=e^{13}
\end{cases}
\]
by a uniform discrete subgroup.
Since $M$ is a nilmanifold, we know by Nomizu theorem that the de Rham cohomology of $M$ is isomorphic to the cohomology of Lie algebra $\mathfrak{g}$ of $G$, i.e., for all degrees $k$ he have the isomorphism
\[
H^k_{dR}(M)\cong H^k(\mathfrak{g}),
\]
therefore $b^1=4$, $b^2=9$ and $b^3=12$.
Endow $M$ with the left invariant almost complex structure $J$ given by
\[
\eta^1=e^1+ie^4,\ \ \ \eta^2=e^2+ie^5,\ \ \ \eta^3=e^3+ie^6.
\]
The complex structure equations are
\[
\begin{cases}
d\eta^1=0,\\
d\eta^2=\frac{i}4\left(\eta^{12}+\eta^{1\c2}-\eta^{2\c1}+\eta^{\c1\c2}\right),\\
d\eta^3=\frac{i}4\left(\eta^{13}+\eta^{1\c3}-\eta^{3\c1}+\eta^{\c1\c3}\right).
\end{cases}
\]
Choose the left invariant almost K\"ahler metric given by
\[
\omega=e^{14}+e^{25}+e^{36}.
\]
Note that the form
\[
e^{145}=\frac14\eta^{1\c1}\wedge(\eta^2-\eta^{\c2})=\frac14\left(-\eta^{12\c1}-\eta^{1\c1\c2}\right)
\]
is harmonic, and since
\[
de^{45}=-e^{124}=\frac{i}4\eta^{1\c1}\wedge(\eta^2+\eta^{\c2})=\frac{i}4\left(-\eta^{12\c1}+\eta^{1\c1\c2}\right),
\]
its de Rham class satisfies
\[
[\eta^{12\c1}+\eta^{1\c1\c2}]=[\eta^{12\c1}+\eta^{1\c1\c2}+4ide^{45}]=2[\eta^{12\c1}]\in H^{2,1}_J,
\]
while its harmonic representative is not of bidegree $(2,1)$, therefore the inclusion
\[
j:\H^{2,1}_d\hookrightarrow H^{2,1}_J
\]
is not surjective and $h^{2,1}_J>h^{2,1}_d$. With similar computations one finds
\[
[\eta^{13\c1}+\eta^{1\c1\c3}]=[\eta^{13\c1}+\eta^{1\c1\c3}+4ide^{46}]=2[\eta^{13\c1}]\in H^{2,1}_J.
\]
By Proposition \ref{prop-isom-cohom-inv} and linear computations, one can show that $h^{1,0}_d=1$, $h^{2,0}_d=0$, $h^{1,1}_d=6$, $h^{2,1}_d=3$ and $h^{1,0}_J=1$, $h^{2,0}_J=0$, $h^{1,1}_J=6$, $h^{2,1}_J=5$, $h^{2,2}_J=7$. Note that $h^{1,1}_d=h^{1,1}_J$ is guaranteed by Corollary \ref{cor-nilm}, and $h^{1,1}_J=6<7=h^{2,2}_J$ shows that the inequality in Corollary \ref{cor-ineq-serre} can be strict.

Let us now investigate the value of $h^{2,1}_d$ for different metrics.
Let $\omega'$ be the fundamental form of any left invariant almost Hermitian metric. Then we can write
\[
\omega'=\frac{i}2(\psi^{1\c1}+\psi^{2\c2}+\psi^{3\c3})
\]
where
\begin{align*}
&\psi^1=A\eta^1,\\
&\psi^2=B\eta^1+C\eta^2,\\
&\psi^3=D\eta^1+E\eta^2+F\eta^3,
\end{align*}
with $A,C,F\in\R\setminus\{0\}$ and $B,D,E\in\C$. It is straightforward to check that $d\omega'=0$ iff
$B,D,E\in \R$.

For any almost Hermitian $\omega'$as above we compute
\begin{align*}
&d\psi^{12\c1}=d\psi^{12\c2}=d\psi^{13\c1}=0,\\
&d\psi^{12\c3}=ACFd\eta^{12\c3}=-ACFd\eta^{13\c2}=-d\psi^{13\c2},\\
&d\psi^{13\c3}=AF(E-\c{E})d\eta^{12\c3},
\end{align*}
while $d\psi^{23\c1},d\psi^{23\c2},d\psi^{23\c3},d\eta^{12\c3}$ can be shown to be linearly independent.
Let
\[
\gamma=\sum H_{ij\c{k}}\psi^{ij\c{k}}
\]
be any left invariant $(2,1)$-form. It follows that $\gamma$ is closed iff
\begin{align*}
&C H_{12\c3}-C H_{13\c2}+(E-\c{E}) H_{13\c3}=0,\\ 
&H_{12\c1}=H_{12\c2}=H_{13\c1}\in\C,\\ 
&H_{23\c1}=H_{23\c2}=H_{23\c3}=0,
\end{align*}
while $\gamma$ is harmonic iff
\begin{align*}
&C H_{12\c3}-C H_{13\c2}+(E-\c{E}) H_{13\c3}=0,\\ 
&C H_{12\c3}-C H_{13\c2}-(E-\c{E}) H_{12\c2}=0,\\ 
&H_{12\c1}=H_{13\c1}=H_{23\c1}=H_{23\c2}=H_{23\c3}=0.
\end{align*}
By Proposition \ref{prop-isom-cohom-inv} every harmonic form is left invariant,
therefore the space of harmonic $(2,1)$-forms is
\[
\H^{2,1}_d=\Span_\C<\psi^{12\c3}+\psi^{13\c2},\psi^{12\c2},\psi^{13\c3}>.
\]
if $E\in\R$, and
\[
\H^{2,1}_d=\Span_\C<\psi^{12\c3}+\psi^{13\c2},\psi^{12\c2}-\psi^{13\c3}+\frac{E-\c{E}}{C}\psi^{12\c3}>
\]
if $E\in\C\setminus\R$. Note that if $\omega'$ is almost K\"ahler, then $h^{2,1}_d=3$.
Summing up, we have the following result.
\begin{proposition}\label{prop ex 21}
For any left invariant almost K\"ahler metric $(g,\omega)$ on $(M,J)$  the natural injection
\[
j:\H^{2,1}_d\hookrightarrow H^{2,1}_J
\]
is not surjective, i.e., $5=h^{2,1}_J>h^{2,1}_d=3$. Moreover, there are no left invariant almost Hermitian metrics for which $h^{2,1}_d$ takes a larger value.
\end{proposition}

Note that by Proposition \ref{prop ex 21}, Lemma \ref{lemma-p0 kahler harm} and Corollary \ref{cor-nilm}, for any $(p,q)$, $h^{p,q}_d$ does not depend on the choice of the left invariant almost K\"ahler metric, and there are no left invariant almost Hermitian metrics for which $h^{p,q}_d$ takes a larger value.\qed

\subsection{$4$-torus with non left invariant almost complex structure}
The following example shows that, in Corollary \ref{cor-nilm}, the assumption of the almost K\"ahler structure $(J,g,\omega)$ being left invariant is necessary for $h^+_J=h^{1,1}_J=b^-+1$. In fact, on the 4-torus we build a non left invariant almost complex structure $J$, compatible with a left invariant almost K\"ahler form $\omega$, such that $h^+_J=h^{1,1}_J=5>4=b^-+1=h^{1,1}_d$.

Let $T^4$ be the 4-torus $\Z^4\backslash \R^4$ with coordinates $x^1,\dots,x^4$. Let $\alpha=\alpha(x^1,x^2)>0$ be a positive function on the torus, and define the almost complex structure $J$ given by
\[
\phi^1=dx^1+i\alpha dx^3,\ \ \ \phi^2=dx^2+idx^4,
\]
being a global coframe of $(1,0)$-forms. The structure equations are
\[
d\phi^1=i\alpha_1dx^{13}+i\alpha_2dx^{23},\ \ \  d\phi^2=0,
\]
where $\alpha_i:=\de{\alpha}{x^i}$, therefore $J$ is integrable iff $\alpha_2=0$.
We set
\begin{align*}
&v^1=\frac1\alpha dx^{12}-dx^{34},\\
&v^2=\frac1\alpha dx^{14}-dx^{23},\\
&w^1=dx^{13},\\
&w^2=\frac1\alpha dx^{24},\\
&w^3=\frac1\alpha dx^{12}+dx^{34},\\
&w^4=\frac1\alpha dx^{14}+dx^{23}.
\end{align*}
Note that $v^1,v^2$ span $J$ anti invariant forms, while $w^1,\dots,w^4$ span $J$ invariant forms. 

Let us first prove that $h^-_J=1$. Let $\psi=Av^1+\alpha Bv^2$ be any $J$ anti invariant form, with $A,B$ smooth functions on the torus, and compute $d\psi$,
\begin{align*}
d\psi=&+dx^{123}\left(\frac1\alpha A_3-(\alpha B)_1\right)+dx^{124}\left(\frac1\alpha A_4-B_2\right)+\\
&-dx^{134}(A_1+B_3)-dx^{234}(A_2+\alpha B_4).
\end{align*}
Therefore $d\psi=0$ iff
\begin{align*}
&\frac1\alpha A_4-B_2=0,\\
&\frac1\alpha A_2+B_4=0,\\
&A_3-\alpha(\alpha B)_1=0,\\
&A_1+B_3=0.
\end{align*}
From the first two equations we get
\[
0=\frac1\alpha A_{44}-B_{24}=\frac1\alpha A_{44}+\frac1\alpha A_{22}+\left(\frac1\alpha\right)_2A_2,
\]
thus $A=A(x^1,x^3)$ by the maximum principle for second order strongly elliptic operators. In fact, we can consider the projection $p:T^4\to T^2$ where $T^2$ is the 2-torus with coordinates $(x^1,x^3)$ and prove that $A,B$ are constant on the fiber $p^{-1}(x^1,x^3)$ which is a 2-torus with coordinates $(x^2,x^4)$. Since $A=A(x^1,x^3)$, from the first two equations it follows $B=B(x^1,x^3)$. From the last two equations we get
\[
0=(\alpha B)_{33}+\alpha A_{13}=(\alpha B)_{33}+\alpha^2(\alpha B)_{11}+\alpha\alpha_1(\alpha B)_1,
\]
thus $\alpha B=\alpha B(x^2,x^4)$ again by the maximum principle. Since neither $\alpha$ nor $B$ depend on $x^4$, then $\alpha B=\alpha B(x^2)$. Moreover, since $\alpha B=\alpha B(x^2)$ and $\alpha=\alpha(x^1,x^2)$, then $B=B(x^1)$. In the following, we choose
\[
\alpha=\sin(2\pi(x^1+x^2))+2,
\]
which implies $B=0$. In fact, from $(\alpha B)_1=0$, we get
\[
B_1=-B\frac{2\pi\cos(2\pi(x^1+x^2))}{\sin(2\pi(x^1+x^2))+2},
\] 
and since $B_1=B_1(x^1)$ we require $B=0$. Then, from the last two equations it follows that $A$ is constant. Summing up, for our choice of $\alpha$ we obtain $h^-_J=1$ and
\[
H^2_{J-}=\C<[v^1]>.
\]

Since $b^2=h^+_J+h^-_J$, we get $h^+_J=5$. We claim that 
\[
H^2_{J+}=\C<[w^1],[\alpha w^2],[w^3],[\alpha w^4+\alpha w^1],[w^4+w^2]>.
\]
Clearly the representatives are $J$ invariant and closed, therefore we only have to prove that the cohomology classes are linearly independent. Assume that a linear combination of the representatives are exact, i.e.,
\[
Aw^1+B\alpha w^2+Cw^3+D(\alpha w^4+\alpha w^1)+E(w^4+w^2)=d\eta,
\]
and compute the following $L^2$ inner products, taken with respect to the standard metric
\[
g=dx^1\otimes dx^1+dx^2\otimes dx^2+dx^3\otimes dx^3+dx^4\otimes dx^4.
\]
Denote by $\la\cdot,\cdot\ra$ the $L^2$ Hermitian inner product of $(p,q)$-forms on $T^4$ with respect to the metric $g$.
Since $dx^{ij}$ is harmonic with respect to $g$, and harmonic forms are orthogonal to exact ones, we get
\begin{align*}
&0=\la d\eta,dx^{13}\ra=A\int_{T^4}dx^{1234}+D\int_{T^4}\alpha dx^{1234}=A+2D,\\
&0=\la d\eta,dx^{14}\ra=D\int_{T^4}dx^{1234}+E\int_{T^4}\frac1\alpha dx^{1234}=D+\frac1{\sqrt3} E,\\
&0=\la d\eta,dx^{23}\ra=D\int_{T^4}\alpha dx^{1234}+E\int_{T^4}dx^{1234}=2D+E,\\
&0=\la d\eta,dx^{24}\ra=B\int_{T^4}dx^{1234}+E\int_{T^4}\frac1\alpha dx^{1234}=B+\frac1{\sqrt3} E,\\
&0=\la d\eta,dx^{34}\ra=C\int_{T^4}dx^{1234}=C,
\end{align*}
therefore $A=B=C=D=E=0$, proving our claim.\qed

\section{Compact almost Hermitian manifolds of dimension $4n$}\label{sec 4n}
Let $(M^{4n},J,g,\omega)$ be a $4n$-dimensional compact almost Hermitian manifold. Here we study the properties of $J$ of being \Cp\ and \Cf\ in degree $2n$, which always hold true when $n=1$ by \cite[Theorem 2.3]{DLZ}. Note that both the Hodge $*$ operator and the almost complex structure $J$ are involutions on the space of $2n$-differential forms, i.e., $*^2=J^2=1$, therefore we have the following decompositions of $2n$-forms:
\[
A^{2n}=A^{2n}_{g+}\oplus A^{2n}_{g-},
\]
and
\[
A^{2n}=A^{2n}_{J+}\oplus A^{2n}_{J-},
\]
where
\[
A^{2n}_{g\pm}:=\{\alpha\in A^{2n}\,|\,*\alpha=\pm\alpha\},
\]
and
\[
A^{2n}_{J\pm}:=\{\alpha\in A^{2n}\,|\,J\alpha=\pm\alpha\}.
\]
Since $*$ and $\Delta_d$ commute, we easily derive
\[
\H^{2n}_d=\H^{2n}_{g+}\oplus \H^{2n}_{g-},
\]
where
\[
\H^{2n}_{g\pm}:=\H^{2n}_d\cap A^{2n}_{g\pm}.
\]
Via the isomorphism between harmonic forms and de Rham cohomology, it is easy to see that
\[
\H^{2n}_{g+}\cong H^{2n}_{g+},
\]
and
\[
\H^{2n}_{g-}\cong H^{2n}_{g-},
\]
therefore the decomposition of harmonic forms has a cohomological counterpart as
\[
H^{2n}_{dR}=H^{2n}_{g+}\oplus H^{2n}_{g-},
\]
where
\[
H^{2n}_{g\pm}:=\{a\in H^{2n}_{dR}\,|\,a=[\alpha], *\alpha=\pm\alpha\}.
\]
A first observation is that in general $J$ is not \Cp\ nor \Cf\ in degree $2n$ on $M^{4n}$. Indeed, we have the following example.

\begin{example}\label{example fino tomassini}
Starting with \cite[Example 3.3]{FT}, we will construct a non integrable almost complex structure on an $8$-dimensional nilmanifold which is not \Cp\ nor \Cf\ in degree $4$. Consider the $8$-dimensional nilmanifold $M$, given by a compact quotient of the
8-dimensional real nilpotent Lie group with structure equations
\[
\begin{cases}
de^j=0,\ j=1,2,3,4,7,8,\\
de^5=e^{12},\\
de^6=e^{13}
\end{cases}
\]
by a uniform discrete subgroup. Endow $M$ with the left invariant almost complex structure $J$ given by
\[
\eta^1=e^1+ie^2,\ \ \ \eta^2=e^3+ie^4,\ \ \ \eta^3=e^5+ie^6,\ \ \ \eta^4=e^7+ie^8.
\]
The complex structure equations are
\[
\begin{cases}
d\eta^j=0,\ j=1,2,4,\\
d\eta^3=\frac{i}2\eta^{1\c1}+\frac{i}4\left(\eta^{12}+\eta^{1\c2}-\eta^{2\c1}+\eta^{\c1\c2}\right).
\end{cases}
\]
Note that $J$ is not \Cp\ in degree $4$, indeed
\begin{align*}
H^{4}_{J+}\ni
\left[\frac{i}{4}(\eta^{1\c2}+\eta^{\c12})\wedge\eta^{4\c4}\right]&=\left[(e^{24}+e^{13})\wedge e^{78}\right]\\
&=\left[(e^{24}-e^{13})\wedge e^{78} \right]=\left[-\frac{i}{4}(\eta^{12}+\eta^{\c1\c2})\wedge\eta^{4\c4}\right]\in H^{4}_{J-}.
\end{align*}
Moreover, $J$ is not \Cf\ in degree $4$. To prove this, let us consider the following $d$-harmonic form
\[
e^{2478}=\frac{i}{8}\left(\eta^{1\c2}+\eta^{\c12}-\eta^{12}-\eta^{\c1\c2}\right)\wedge\eta^{4\c4}
\]
We want to show that $e^{2478}$ is wedge product orthogonal to every $d$-closed 4-form which is either $J$ invariant or $J$ anti invariant. In this way, we are proving that the de Rham class $[e^{2478}]$ is cup product orthogonal to both the spaces $H^{4}_{J+}$ and $H^{4}_{J-}$. Let $\alpha$ be a real $J$ invariant 4-form and assume that $\alpha\wedge e^{2478}\ne0$. It turns out that $\alpha$ must be of the form
\begin{align*}
\alpha&=\frac{i}2\left(f\eta^{1\c2}+\c{f}\eta^{\c12}\right)\wedge\eta^{3\c3}\\
&=\left(f(e^{13}+e^{24}-ie^{14}+ie^{23})+\c{f}(e^{13}+e^{24}+ie^{14}-ie^{23})\right)\wedge e^{56},
\end{align*}
with $f$ a smooth complex valued function, and from this we easily compute $d\alpha\ne0$. Therefore, if $\alpha$ is a real $d$-closed $J$ invariant 4-form, then $\alpha\wedge e^{2478}=0$. Arguing in the same way, if $\alpha$ is a real $J$ anti invariant 4-form such that $\alpha\wedge  e^{2478}\ne0$, then
\begin{align*}
\alpha&=\frac{i}2\left(f\eta^{12}+\c{f}\eta^{\c1\c2}\right)\wedge\eta^{3\c3}\\
&=\left(f(e^{13}-e^{24}+ie^{14}+ie^{23})+\c{f}(e^{13}-e^{24}-ie^{14}-ie^{23})\right)\wedge e^{56},
\end{align*}
with $f$ a smooth complex valued function, and $d\alpha\ne0$. This proves our claim.
\end{example}

However, following the proof of \cite[Theorem 2.3]{DLZ} and working separately on the two spaces $H^{2n}_{g+}$ and $H^{2n}_{g-}$ introduced before, we can prove that $J$ is \lq\lq\Cp\ in degree $2n$" for these two spaces $H^{2n}_{g+}$ and $H^{2n}_{g-}$.
Set
\[
H^{2n}_{g\pm,J+}:=\{a\in H^{2n}_{dR}\,|\,a=[\alpha], *\alpha=\pm\alpha, J\alpha=+\alpha\}
\]
and
\[
H^{2n}_{g\pm,J-}:=\{a\in H^{2n}_{dR}\,|\,a=[\alpha], *\alpha=\pm\alpha, J\alpha=-\alpha\}.
\]
\begin{theorem}\label{thm 4n pure}
Let $(M^{4n},J,g)$ be a $4n$-dimensional compact almost Hermitian manifold. Then
\[
H^{2n}_{g+,J+}\cap H^{2n}_{g+,J-}=\{[0]\},
\]
and
\[
H^{2n}_{g-,J+}\cap H^{2n}_{g-,J-}=\{[0]\}.
\]
\end{theorem}
\begin{proof}
Let $a\in H^{2n}_{g\pm,J+}\cap H^{2n}_{g\pm,J-}$, and choose two forms $\alpha'\in A^{2n}_{g\pm}\cap A^{2n}_{J+}$ and $\alpha''\in A^{2n}_{g\pm}\cap A^{2n}_{J-}$ such that $a=[\alpha']=[\alpha'']$. Then
\[
a\cup a=\int_M\alpha'\wedge\alpha''=0
\]
for bidegree reasons, but we also have
\[
a\cup a=\int_M\alpha'\wedge\alpha'=\pm\int_M\alpha'\wedge*\alpha'=\pm\lVert\alpha'\rVert_{L^2},
\]
therefore $\alpha'=0$ and $a=[0]$.
\end{proof}

However, note that in general it might happen that
\[
H^{2n}_{g+,J+}\oplus H^{2n}_{g+,J-}\underset{\ne}{\subset}H^{2n}_{g+}
\]
and
\[
H^{2n}_{g-,J+}\oplus H^{2n}_{g-,J-}\underset{\ne}{\subset}H^{2n}_{g-}.
\]
To see this, assume on the contrary that the spaces $H^{2n}_{g\pm}$ actually decompose into these direct sums. It would follow that
\begin{align*}
H^{2n}_{dR}=H^{2n}_{g+}\oplus H^{2n}_{g-}=H^{2n}_{g+,J+}\oplus H^{2n}_{g+,J-}\oplus H^{2n}_{g-,J+}\oplus H^{2n}_{g-,J-}\subseteq H^{2n}_{J+}+H^{2n}_{J-}\subseteq H^{2n}_{dR}
\end{align*}
which implies that $J$ is \Cf\ in degree $2n$. It suffices to consider Example \ref{example fino tomassini} to have an example where this cannot happen.

\end{document}